\documentclass[11pt]{article}
\usepackage[margin=1in]{geometry} 
\usepackage{amssymb,amsfonts,amsmath,bbm,mathrsfs,stmaryrd,mathtools,mathabx}
\usepackage{xcolor}
\usepackage{url}

\usepackage[shortlabels]{enumitem}
\usepackage{tensor}

\usepackage[colorlinks,
linkcolor=black!75!red,
citecolor=blue,
pdftitle={},
pdfproducer={pdfLaTeX},
pdfpagemode=None,
bookmarksopen=true,
bookmarksnumbered=true,
backref=page]{hyperref}

\usepackage{tikz}
\usetikzlibrary{arrows,calc,decorations.pathreplacing,decorations.markings,intersections,shapes.geometric,through,fit,shapes.symbols,positioning,decorations.pathmorphing}

\usepackage{braket}

\usepackage[amsmath,thmmarks,hyperref]{ntheorem}
\usepackage{cleveref}

\creflabelformat{enumi}{#2#1#3}

\crefname{section}{Section}{Sections}
\crefformat{section}{#2Section~#1#3} 
\Crefformat{section}{#2Section~#1#3} 

\crefname{subsection}{\S}{\S\S}
\AtBeginDocument{%
  \crefformat{subsection}{#2\S#1#3}%
  \Crefformat{subsection}{#2\S#1#3}% 
}

\crefname{subsubsection}{\S}{\S\S}
\AtBeginDocument{%
  \crefformat{subsubsection}{#2\S#1#3}%
  \Crefformat{subsubsection}{#2\S#1#3}% 
}

\theoremstyle{plain}

\newtheorem{lemma}{Lemma}[section]
\newtheorem{proposition}[lemma]{Proposition}
\newtheorem{corollary}[lemma]{Corollary}
\newtheorem{theorem}[lemma]{Theorem}

\theoremstyle{plain}
\theoremnumbering{Alph}
\newtheorem{theoremN}{Theorem}

\theoremstyle{plain}
\theorembodyfont{\upshape}
\theoremsymbol{\ensuremath{\blacklozenge}}

\newtheorem{definition}[lemma]{Definition}
\newtheorem{example}[lemma]{Example}

\newtheorem{remark}[lemma]{Remark}
\newtheorem{remarks}[lemma]{Remarks}

\newtheorem{conventions}[lemma]{Conventions}

\crefname{definition}{definition}{definitions}
\crefformat{definition}{#2definition~#1#3} 
\Crefformat{definition}{#2Definition~#1#3} 

\crefname{ex}{example}{examples}
\crefformat{example}{#2example~#1#3} 
\Crefformat{example}{#2Example~#1#3} 

\crefname{exs}{example}{examples}
\crefformat{examples}{#2example~#1#3} 
\Crefformat{examples}{#2Example~#1#3} 

\crefname{remark}{remark}{remarks}
\crefformat{remark}{#2remark~#1#3} 
\Crefformat{remark}{#2Remark~#1#3} 

\crefname{remarks}{remark}{remarks}
\crefformat{remarks}{#2remark~#1#3} 
\Crefformat{remarks}{#2Remark~#1#3} 

\crefname{convention}{convention}{conventions}
\crefformat{convention}{#2convention~#1#3} 
\Crefformat{convention}{#2Convention~#1#3} 

\crefname{conventions}{conventions}{conventions}
\crefformat{conventions}{#2conventions~#1#3} 
\Crefformat{conventions}{#2Conventions~#1#3} 

\crefname{notation}{notation}{notations}
\crefformat{notation}{#2notation~#1#3} 
\Crefformat{notation}{#2Notation~#1#3} 

\crefname{table}{table}{tables}
\crefformat{table}{#2table~#1#3} 
\Crefformat{table}{#2Table~#1#3}

\crefname{lemma}{lemma}{lemmas}
\crefformat{lemma}{#2lemma~#1#3} 
\Crefformat{lemma}{#2Lemma~#1#3} 

\crefname{proposition}{proposition}{propositions}
\crefformat{proposition}{#2proposition~#1#3} 
\Crefformat{proposition}{#2Proposition~#1#3} 

\crefname{corollary}{corollary}{corollaries}
\crefformat{corollary}{#2corollary~#1#3} 
\Crefformat{corollary}{#2Corollary~#1#3} 

\crefname{theorem}{theorem}{theorems}
\crefformat{theorem}{#2theorem~#1#3} 
\Crefformat{theorem}{#2Theorem~#1#3} 

\crefname{theoremN}{theorem}{theorems}
\crefformat{theoremN}{#2theorem~#1#3} 
\Crefformat{theoremN}{#2Theorem~#1#3} 

\crefname{enumi}{}{}
\crefformat{enumi}{#2#1#3}
\Crefformat{enumi}{#2#1#3}

\crefname{assumption}{assumption}{Assumptions}
\crefformat{assumption}{#2assumption~#1#3} 
\Crefformat{assumption}{#2Assumption~#1#3} 

\crefname{construction}{construction}{Constructions}
\crefformat{construction}{#2construction~#1#3} 
\Crefformat{construction}{#2Construction~#1#3} 

\crefname{equation}{}{}
\crefformat{equation}{(#2#1#3)} 
\Crefformat{equation}{(#2#1#3)}

\numberwithin{equation}{section}

\theoremstyle{nonumberplain}
\theoremsymbol{\ensuremath{\blacksquare}}

\newtheorem{proof}{Proof}
\newcommand\pf[1]{\newtheorem{#1}{Proof of \Cref{#1}}}

\newcommand\altpf[2]{\newtheorem{#1}{#2}}

\newcommand\bC{{\mathbb C}}

\newcommand\bG{{\mathbb G}}
\newcommand\bH{{\mathbb H}}

\newcommand\bK{{\mathbb K}}
\newcommand\bL{{\mathbb L}}

\newcommand\bR{{\mathbb R}}
\newcommand\bS{{\mathbb S}}
\newcommand\bT{{\mathbb T}}
\newcommand\bU{{\mathbb U}}

\newcommand\bZ{{\mathbb Z}}

\newcommand\cC{{\mathcal C}}

\newcommand\cF{{\mathcal F}}

\newcommand\cL{{\mathcal L}}

\newcommand\fG{{\mathfrak G}}

\newcommand\fM{{\mathfrak M}}

\newcommand\fn{{\mathfrak n}}

\DeclareMathOperator{\id}{id}

\DeclareMathOperator{\Ind}{\mathrm{Ind}}

\DeclareMathOperator{\spn}{\mathrm{span}}
\DeclareMathOperator{\Irr}{\mathrm{Irr}}

\newcommand{\cat}[1]{\textsc{#1}}

\newcommand\spr[1]{\cite[\href{https://stacks.math.columbia.edu/tag/#1}{Tag {#1}}]{stacks-project}}
\newcommand{\qedhere}{\mbox{}\hfill\ensuremath{\blacksquare}}

\newcommand{\xrightarrowdbl}[2][]{%
  \xrightarrow[#1]{#2}\mathrel{\mkern-14mu}\rightarrow
}

\title{Factorization and boundedness for representations of locally compact groups on topological vector spaces}
\author{Alexandru Chirvasitu}

\begin{document}

\date{}

\newcommand{\Addresses}{{% additional braces for segregating \footnotesize
  \bigskip
  \footnotesize

  \textsc{Department of Mathematics, University at Buffalo}
  \par\nopagebreak
  \textsc{Buffalo, NY 14260-2900, USA}  
  \par\nopagebreak
  \textit{E-mail address}: \texttt{achirvas@buffalo.edu}

  % % \medskip
  % % 
  % % \textsc{Department of Mathematics, INSTITUTION}
  % % \par\nopagebreak
  % % \textsc{ADDRESS}
  % % \par\nopagebreak
  % % \textit{E-mail address}: \texttt{??}
  % % 

}}

\maketitle

\begin{abstract}
  We (a) prove that continuous morphisms from locally compact groups to locally exponential (possibly infinite-dimensional) Lie groups factor through Lie quotients, recovering a result of Shtern's on factoring norm-continuous representations on Banach spaces; (b) characterize the maximal almost-periodicity of the identity component $\mathbb{G}_0\le \mathbb{G}$ of a locally compact group in terms of sufficiently discriminating families of continuous functions on $\mathbb{G}$ valued in Hausdorff spaces generalizing an analogous result by Kadison-Singer; (c) apply that characterization to recover the von Neumann kernel of $\mathbb{G}_0$ as the joint kernel of all appropriately bounded and continuous $\mathbb{G}$-representations on topological vector spaces extending Kallman's parallel statement for unitary representations, and (d) provide large classes of complete locally convex topological vector spaces (e.g. arbitrary products of Fr\'echet spaces) with the property that compact-group representations thereon whose vectors all have finite-dimensional orbits decompose as finite sums of isotypic components. This last result specializes to one of Hofmann-Morris on representations on products of real lines. 
\end{abstract}


\vspace{.5cm}


%\tableofcontents

%%%%%%%%%%%%%%%%%%%%%%%%%%%%%%%%%%%%%%%%%%%%%%%%%%%%%%%%%%%%%%%%%%%%%%%%%%%%%
%%%%%%%%%%%%%%%%%%%%%%%%%%%%%%%%%%%%%%%%%%%%%%%%%%%%%%%%%%%%%%%%%%%%%%%%%%%%%
\section*{Introduction}

The several themes the paper touches on all ultimately emerge from and point back to the study of norm-continuous (or uniformly continuous) representations of locally compact groups. The precise setting differs among sources: some \cite{zbMATH03072119,klm_unif,sing_unif} focus on unitary representations while others \cite{gurar_unif,lv_unif,zbMATH05628052} discuss representations on Banach spaces. Similarly, the framework is occasionally limited to particular classes of groups: Lie (finite dimensional), or connected, or both. The common pattern is the constrained nature and structure of such representations.

As a case in point, consider the fact \cite[Corollary 2]{zbMATH05628052} that a norm-continuous representation of a compact group $\bG$ on a Banach space $E$ decomposes as a sum of finitely many {\it isotypic components} \cite[Definition 4.21]{hm4}:
\begin{equation*}
  E=\bigoplus_{i=1}^n E_{\alpha_i},\quad \alpha_i\in\Irr(\bG):=\text{classes of irreducible $\bG$-representations},
\end{equation*}
where $E_{\alpha}$, $\alpha\in\Irr(\bG)$ is the set of vectors generating finite-dimensional subrepresentations that are direct sums of copies of $\alpha$. The unitary case also follows from \cite[main result stated on p.257]{klm_unif}, while \cite[Theorem 2.10]{chirvasitu2023quantum} contains a number of other characterizations (again in the Banach setup). In \cite{zbMATH05628052}, the key result delivering this finiteness is the factorization \cite[Theorem 1]{zbMATH05628052} of every norm-continuous representation, of any locally compact group whatsoever, through a Lie-group quotient. \Cref{th:lcg2lie} below records the observation that this goes through rather generally:

\begin{theoremN}
  A continuous morphism $\bG\to \bU$ from a locally compact group to a topological group factors through a Lie quotient provided all compact subgroups of $\bU$ are finite-dimensional Lie.

  In particular, this applies to {\it locally exponential} Lie groups $\bU$.  \qedhere
\end{theoremN}

The phrase {\it Lie group} is construed as in \cite[Definition II.1.1]{neeb-lc}: topological groups locally modeled on (possibly infinite-dimensional) {\it locally convex} \cite[\S 18.1]{k_tvs-1} topological vector spaces. Recall \cite[Definition IV.1.1]{neeb-lc} that the {\it locally exponential} Lie groups are those locally diffeomorphic to their Lie algebra via an exponential function, and that they do indeed \cite[Theorem IV.3.16]{neeb-lc} satisfy the compact-subgroup. \cite[Theorem 1]{zbMATH05628052} then follows by further specializing to $\bU:=GL(E)$ for a Banach space $E$. 

In the course of investigating norm-continuous unitary representations of a locally compact {\it connected} group $\bG$, \cite[Lemmas 1 and 2]{klm_unif} (in a slight rephrasing) note that the joint kernel of those representations is the smallest closed normal subgroup $\bH\trianglelefteq \bG$ for which
\begin{equation}\label{eq:rdl}
  \bG/\bH\cong \bR^d\times \bL
  ,\quad
  d\in \bZ_{\ge 0}
  ,\quad
  \bL\text{ compact connected}.
\end{equation}
In other words, $\bH$ is precisely the {\it von Neumann kernel} \cite[\S 0]{zbMATH03680105} of $\bG$: the intersection of the kernels of finite-dimensional unitary $\bG$-representations. Those results in turn rely heavily on the characterization \cite[Lemma 1 and Corollary 4]{zbMATH03072119} of groups of the form \Cref{eq:rdl} in terms of certain families of conjugacy-invariant functions thereon. This too generalizes (\Cref{th:upperhemi}) so as to allow extending the discussion to representations on Banach spaces and beyond:

\begin{theoremN}
  The identity component $\bG_0\le \bG$ of a locally compact group is of the form \Cref{eq:rdl} if and only if there is a family of continuous functions $\bG\xrightarrow{\varphi}X_{\varphi}$ into Hausdorff topological spaces such that
  \begin{itemize}
  \item for every $1\ne g\in \bG$ there is some $\varphi$ with $\varphi(g)\ne \varphi(1)$;
  \item and for every $\varphi$ the set-valued map
    \begin{equation*}
      \widetilde{\varphi}:=
      \bG\ni g
      \xmapsto{\quad}
      \varphi\left(\text{$\bG_0$-conjugacy class of $g$}\right)
      \in 2^{X_{\varphi}}
    \end{equation*}
    is {\it upper hemicontinuous} \cite[Definition 17.2]{zbMATH05265624} at $1\in \bG$:
    \begin{equation*}
      \forall\text{ neighborhood }U\ni \varphi(1)
      ,\quad
      \exists\text{ neighborhood }V\ni 1\in \bG
      \text{ with }
      g\in V\xRightarrow{\quad}\widetilde{\varphi}(g)\subseteq U. 
    \end{equation*}
    \qedhere
  \end{itemize}
\end{theoremN}

There are thus (\Cref{th:enoughreps} and \Cref{cor:vnker-tvs}) representation-theoretic characterizations of locally compact groups with identity components of the form \Cref{eq:rdl} via broader classes of topological vector spaces (than Banach spaces, say)

\begin{theoremN}\label{thn:vnk}
  Let $\bG$ be a locally compact group. 
  
  \begin{enumerate}[(1)]
  \item\label{item:thn:vnk:g0} The identity component of $\bG$ is of the form \Cref{eq:rdl} if and only if $\bG$ has a separating family of representations $\bG\xrightarrow{\varphi}GL(E_{\varphi})$ on topological vector spaces such that
    \begin{itemize}
    \item $\varphi$ is continuous for the topology of uniform convergence on members of a {\it covering bornology} \cite[\S III.1.1, Definition 1]{bourb_tvs} $\fM\subseteq 2^{E_{\varphi}}$;

    \item the family $\varphi(\bG_0)\subset GL(E_{\varphi})$ is equicontinuous;

    \item and $\varphi(\bG_0)M\in \fM$ for all $M\in \fM$. 
    \end{itemize}

  \item\label{item:thn:vnk:vnk} Consequently, the von Neumann kernel of $\bG_0\le \bG$ is precisely the common kernel of all representations $\varphi$ as in \Cref{item:thn:vnk:g0}.  \qedhere
  \end{enumerate}   
\end{theoremN}

The conditions of \Cref{thn:vnk}\Cref{item:thn:vnk:g0} specialize precisely to norm-continuity and norm-boundedness for representations on Banach spaces or, more generally, on {\it quasi-Banach} spaces \cite[\S A.1]{pvl_disc}. 

\Cref{se:hm4-e48} varies the theme of isotypic-component finiteness, taking \cite[Exercise E4.8]{hm4} as a starting point. That result says that compact group representations on {\it weakly complete} topological vector spaces $E$ (i.e. \cite[Definition A7.8]{hm4} products of real lines equipped with their product topology) decompose as finite sums of isotypic components provided every orbit $\bG v$, $v\in E$ has finite-dimensional span. In the language of \Cref{def:gfin}, a {\it locally $\bG$-finite} representation of a compact group $\bG$ on a weakly complete TVS is {\it globally $\bG$-finite}. 

It is, on the one hand, easy to recover the analogous result (\Cref{pr:hme48-frechet}) for {\it Fr\'echet spaces} $E$ instead (\cite[\S 10]{trev_tvs}: locally convex complete metrizable TVSs). On the other hand, weakly complete spaces will be Fr\'echet only when the family of real-line factors is at most countable. It is natural, then, to try to find a common generalization of the two results: a class of topological vector spaces that naturally houses both weakly complete and Fr\'echet spaces, for which local $\bG$-finiteness entails its global counterpart. \Cref{th:lim-stab-glob} does this in terms of realizing a TVS as a cofiltered limit of well-behaved quotients; we record herer the more straightforwardly-stated consequence (\Cref{cor:lim-stab-glob}):

\begin{theoremN}
  A locally $\bG$-finite representation of a compact group $\bG$ on a product of Fr\'echet spaces is a finite direct sum of isotypic components.  \qedhere
\end{theoremN}

%%%%%%%%%%%%%%%%%%%%%%%%%%%%%%%%%%%%%%%%%%%%%%%%%%%%%%%%%%%%%%%%%%%%%%%%%%%%%
\subsection*{Acknowledgements}

I am grateful for comments, suggestions and pointers to the literature by K. H. Hofmann and K.-H. Neeb. 

This work is partially supported by NSF grant DMS-2001128. 

%%%%%%%%%%%%%%%%%%%%%%%%%%%%%%%%%%%%%%%%%%%%%%%%%%%%%%%%%%%%%%%%%%%%%%%%%%%%%
%%%%%%%%%%%%%%%%%%%%%%%%%%%%%%%%%%%%%%%%%%%%%%%%%%%%%%%%%%%%%%%%%%%%%%%%%%%%%
% % \section{Preliminaries}\label{se.prel}

%%%%%%%%%%%%%%%%%%%%%%%%%%%%%%%%%%%%%%%%%%%%%%%%%%%%%%%%%%%%%%%%%%%%%%%%%%%%%
%%%%%%%%%%%%%%%%%%%%%%%%%%%%%%%%%%%%%%%%%%%%%%%%%%%%%%%%%%%%%%%%%%%%%%%%%%%%%
\section{Factorization through well-behaved groups}\label{se:factlie}

We make fairly free use of terminology pertinent to topological vector spaces (TVSs for short), for which the reader can consult (say) \cite{bourb_tvs,k_tvs-1,k_tvs-2,nb_tvs,rr_tvs,schae_tvs,trev_tvs} and countless others. Where the occasion demands them, the text provides specific references. Topological groups (so TVSs in particular) are assumed Hausdorff by default. 

The same goes for material on {\it uniform spaces}, as covered, say, in \cite[Chapter 7 and onward]{james_unif_1999}, \cite[\S 5]{k_tvs-1}, \cite[Chapter II]{bourb_top_en_1}, etc. In particular, we take for granted notions such as 
\begin{itemize}
\item the topology underlying a uniform structure \cite[\S II.1.2]{bourb_top_en_1};
\item the {\it left} and {\it right uniformities} \cite[\S III.3.1, Definition 1]{bourb_top_en_1} on a topological group;
\item the {\it product} uniformity \cite[\S II.2.6, Definition 4]{bourb_top_en_1} on a Cartesian product of topological spaces;
\item and the {\it completion} \cite[\S II.3.7]{bourb_top_en_1} of a Hausdorff uniform space. 
\end{itemize}

There are references to {\it (co)limits} \cite[\S 11]{ahs}, mostly in categories of TVSs. Of particular importance are
\begin{itemize}
\item {\it cofiltered limits} \spr{04AY}, i.e. limits of diagrams of the form
  \begin{equation*}
    E_i\xrightarrow{\quad\pi_{ji}\quad}E_j,\quad i\ge j\text{ in a directed (or filtered) poset }(I,\le),
  \end{equation*}
  meaning that any two elements are dominated by another;
\item and dually, {\it filtered colimits}, i.e. colimits of diagrams
  \begin{equation*}
    E_i\xrightarrow{\quad\iota_{ji}\quad}E_j,\quad i\le j\in\text{ filtered }(I,\le).
  \end{equation*}
\end{itemize}
Other terms (respectively) are {\it projective} \cite[\S 19.7]{k_tvs-1} and {\it inductive} \cite[\S 19.2]{k_tvs-1} {\it limits}. 

Although phrased in terms of representations on Banach spaces $E$ and hence morphisms into Banach Lie groups $\bU:=GL(E)$, the core of the proof of \cite[Theorem 1]{zbMATH05628052} can be reworked to deliver substantially more. 

\begin{theorem}\label{th:lcg2lie}
  \begin{enumerate}[(1)]
  \item\label{item:th:lcg2lie:cpctislie} Let $\bU$ be a (Hausdorff) topological group whose compact subgroups are finite-dimensional Lie. Any topological-group morphism $\bG\xrightarrow{\varphi}\bU$ from a locally compact group factors through a finite-dimensional Lie-group quotient of $\bG$.

  \item\label{item:th:lcg2lie:partcases} In particular, the conclusion of \Cref{item:th:lcg2lie:cpctislie} holds for {\it locally exponential} \cite[Definition IV.1.1]{neeb-lc} Lie groups $\bU$. 
  \end{enumerate}
\end{theorem}
\begin{proof}  
  Part \Cref{item:th:lcg2lie:partcases} follows immediately from \Cref{item:th:lcg2lie:cpctislie}, since the latter's hypothesis is met \cite[Theorem IV.3.16]{neeb-lc}: (even locally) compact subgroups of locally exponential Lie groups are finite-dimensional Lie. Settling, \Cref{item:th:lcg2lie:cpctislie}, then, is enough. 
  
  As noted in the proof of the already-cited \cite[Theorem 1]{zbMATH05628052}, it is enough handle the restriction of $\varphi$ to a judiciously-chosen open subgroup $\bH\le\bG$, for we would then know that $\varphi$ factors through a quotient $\bG\xrightarrowdbl{}\overline{\bG}$ having a finite-dimensional Lie quotient of $\bH$ as an open subgroup, so that $\overline{\bG}$ is itself Lie.

  We can choose \cite[Lemma 2.3.1]{mz} the aforementioned open $\bH\le \bG$ {\it almost connected} \cite[preceding Theorem 10.89]{hm4}, i.e. with $\bH/\bH_0=\bH/\bG_0$ compact. But then $\bH$ has \cite[\S 4.6, Theorem]{mz} a compact normal subgroup $\bK\trianglelefteq \bH$ with $\bH/\bK$ Lie (finite-dimensional). The (automatically compact) image $\varphi(\bK)\le \bU$ is finite-dimensional Lie by assumption, so $\varphi$ factors through the (finite-dimensional Lie \cite[Theorem 7]{iw}) extension of the Lie group $\bH/\bK$ by the Lie group $\varphi(\bK)\cong \bK/\ker\varphi|_{\bK}$. 
\end{proof}

%\newpage

We record an alternative argument \Cref{th:lcg2lie}\Cref{item:th:lcg2lie:partcases} assuming the Lie algebra $Lie(\bG)$ admits a continuous norm. That assumption, while weakening the result, still holds \cite[Corollary IV.1.10]{neeb-lc} for {\it Banach} Lie groups \cite[Definition IV.I]{neeb-inf}, so the result still recovers (and generalizes) \cite[Theorem 1]{zbMATH05628052} upon specializing $\bU$ to $GL(E)$ for Banach spaces $E$. 

\altpf{pf:th:lcg2lie:alt}{Proof of \Cref{th:lcg2lie}\Cref{item:th:lcg2lie:partcases} when $Lie(\bU)$ admits a continuous norm}
\begin{pf:th:lcg2lie:alt}
  The assumptions on $\bU$ ensure \cite[Problem II.5]{neeb-lc} that it contains {\it no small  subgroups} \cite[2.20.1]{mz}: there are neighborhoods $V\ni 1\in \bU$ containing no non-trivial subgroups. But then
  \begin{itemize}
  \item every subgroup of the preimage $\varphi^{-1}V\subseteq \bG$ will be annihilated by $\varphi$;
  \item which goes for a normal compact subgroup $\bK\trianglelefteq \bG'$ of an open subgroup $\bG'\le \bG$ with $\bG'/\bK$ Lie (and there are such \cite[\S 4.0]{mz});
  \item which means that $\varphi$ factors through the finite-dimensional Lie group obtained as the quotient of $\bG$ by the normal subgroup generated by $\bK$.  \qedhere
  \end{itemize}
\end{pf:th:lcg2lie:alt}

\begin{remarks}\label{res:gurarienotwrong}
  \begin{enumerate}[(1),wide=0pt]
  \item \Cref{th:lcg2lie}\Cref{item:th:lcg2lie:partcases} generalizes \cite[Theorem IV.3.16]{neeb-lc}, which says that (automatically closed) locally compact subgroups of locally-exponential Lie groups are finite-dimensional Lie. The (first) proof given above does of course appeal to \cite[Theorem IV.3.16]{neeb-lc}, but in its weaker form pertaining to {\it compact} groups. 
    
  \item Naturally, the assumed continuous norm on $Lie(\bU)$ in the second proof of \Cref{th:lcg2lie}\Cref{item:th:lcg2lie:partcases} need not induce the topology on that space. In other words, $Lie(\bU)$ need not be {\it normable} \cite[\S 11]{trev_tvs}. The space $H(\Omega)$ \cite[\S 10, Example II]{trev_tvs} of holomorphic functions on an open subset $\Omega\subseteq \bC^n$, with the topology of uniform convergence on compact sets, is \cite[Proposition 34.4]{trev_tvs} a {\it Montel space} in the sense of \cite[Definition 34.2]{trev_tvs}: Hausdorff, {\it locally convex} (\cite[\S 18.1]{k_tvs-1}, \cite[\S II.4.1]{bourb_tvs}), {\it barreled} (\cite[\S 21.2]{k_tvs-2}, \cite[Definition 33.1]{trev_tvs}) and such that every closed {\it bounded} \cite[Definition 14.1]{trev_tvs} subset is compact.
   
    Now, on the one hand normable Montel spaces are finite-dimensional \cite[paragraph following Definition 34.2]{trev_tvs}, so that $H(\Omega)$ cannot be normable if $\Omega$ is non-empty. On the other hand, holomorphic functions on {\it connected} $\Omega$ vanishing on a compact subset $K\subset\Omega$ with non-empty interior must vanish globally, so $\|\cdot\|:=\sup_{K}|\cdot|$ is a continuous norm on $H(\Omega)$.
    
  \item \cite[Footnote 1]{zbMATH05628052} expresses some concerns with the claim, in the course of the proof of \cite[Corollary 2]{gurar_unif}, that the image of a uniformly-continuous representation $\bG\xrightarrow{\varphi} GL(E)$ of a (connected) locally compact group $\bG$ on a Banach space $E$ is a (connected) locally compact group with no small subgroups.
    
    That claim seems to me correct, in the charitable interpretation of equipping the image with the intrinsic topology (i.e. the {\it quotient topology} \cite[\S I.2.4, Example 1)]{bourb_top_en_1} inherited from $\bG$) rather than the subspace topology inherited from $GL(E)$ (the contrast between these two topologies is the flag raised by \cite[Footnote 1]{zbMATH05628052}):
    \begin{itemize}
    \item $\bG/\ker\varphi$ {\it is} locally compact in the quotient topology, the kernel being a closed subgroup;

    \item and it also satisfies the no-small-subgroups condition because its image in the {\it coarser} subspace topology on
      \begin{equation*}
        \varphi\left(\bG/\ker\varphi\right)\subseteq GL(E)
      \end{equation*}
      does: plainly, the no-small-subgroup requirement survives a {\it strengthening} of the topology. 
    \end{itemize}
    For these reasons, the proof of \cite[Corollary 2]{gurar_unif} appears to me to be correct.
  \end{enumerate}
\end{remarks}

The importance in \cite{klm_unif} of working with {\it unitary} representations (as opposed to, say, representations on Banach spaces) and {\it connected} locally compact groups can be discerned by tracing through one of its main sources, \cite{zbMATH03072119}: \cite[Lemmas 1 and 2]{klm_unif} appeal to \cite[Theorem 1 and Corollary 4]{zbMATH03072119} respectively, both of which point back to \cite[Lemma 1]{zbMATH03072119}. The latter is a structure result stating that connected locally compact groups admitting families of continuous functions which
\begin{itemize}
\item are constant along conjugacy classes;
\item and separate non-trivial conjugacy classes from $\{1\}$ 
\end{itemize}
are products of compact (connected) groups and Euclidean groups $(\bR^d,+)$. In turn, the result appeals to the fact \cite[Theorem 3]{zbMATH03067716} that said products are precisely the connected locally compact groups having fundamental systems of conjugacy-invariant compact $1$-neighborhoods.

Per \cite[Corollary 4]{zbMATH03072119}, that structure result applies to connected locally compact groups with enough norm-continuous unitary representations $\bG \xrightarrow{\varphi}\bU(H)$ by applying the aforementioned \cite[Lemma 1]{zbMATH03072119} to the functions of the form $\|\varphi(\bullet)-1\|$. The unitary character of the representation is crucial here in ensuring the invariance of the functions on conjugacy classes. A relaxation of that unitary constraint thus entails a weakening of the conjugacy invariance of the functions in \cite[Lemma 1]{zbMATH03072119}; this is what motivates \Cref{th:upperhemi} below.  

Recall \cite[Definition 17.2]{zbMATH05265624} that a set-valued map $X\xrightarrow{\psi} 2^{Y}$ for topological spaces $X$ and $Y$ (sometimes \cite[Definition 17.1]{zbMATH05265624} referred to as a {\it correspondence} $X\rightsquigarrow Y$) is {\it upper hemicontinuous at $x_0\in X$} if the {\it upper inverse images} (or {\it preimages}) \cite[\S 17.1, following Definition 17.1]{zbMATH05265624} 
\begin{equation*}
  \psi_{-1}(U)
  :=
  \left\{x\in X\ |\ \psi(x)\subseteq U\right\},\quad \text{ neighborhoods }U\supseteq \psi(x_0)
\end{equation*}
are neighborhoods of $x_0$. 

\begin{remark}
  The $(-1)$-subscript notation is borrowed from \cite[\S 1]{zbMATH06329568}, where upper inverse images are termed {\it small} preimages (and {\it hemi}continuity is {\it semi}continuity). 
\end{remark}

\begin{theorem}\label{th:upperhemi}
  Let $\bG$ be a locally compact group with identity component $\bG_0$. The following conditions are equivalent.
  \begin{enumerate}[(a)]
  \item\label{item:th:upperhemi:g0inv} $1\in \bG$ has a fundamental system of $\bG_0$-conjugacy-invariant compact neighborhoods.

  \item\label{item:th:upperhemi:g0rl} $\bG_0\cong \bR^d\times \bL$ for some compact connected $\bL$ and $d\in \bZ_{\ge 0}$. 
    
  \item\label{item:th:upperhemi:openrk} There is an open $\bH\le \bG$ expressible as a cofiltered limit $\bH\cong \varprojlim_i \bH_i$ of (finite-dimensional Lie) quotients, with each $\bH_i$ containing an open subgroup of the form $\bR^d\times \bL_i$ for compact (finite-dimensional Lie) $\bL_i$.

  \item\label{item:th:upperhemi:famrealconst} There is a family of real-valued functions $\bG\xrightarrow{\varphi\in \cF}\bR$, separating $g\ne 1$ from $1$ (i.e. distinguishing element non-triviality) and constant along $\bG_0$-conjugacy classes.

  \item\label{item:th:upperhemi:famrealuhc} Same as \Cref{item:th:upperhemi:famrealconst}, but replacing the conjugacy-class constancy with the formally weaker requirement that 
    \begin{equation*}
      \widetilde{\varphi}:=
      \bG\ni g
      \xmapsto{\quad}
      \varphi\left(\text{$\bG_0$-conjugacy class of $g$}\right)
      \in 2^{\bR}
      ,\quad
      \varphi\in \cF
    \end{equation*}
    be upper hemicontinuous at $1\in \bG$.
    
  \item\label{item:th:upperhemi:famxconst} Analogue of \Cref{item:th:upperhemi:famrealconst}, but with $\varphi\in \cF$ taking values in varying Hausdorff spaces $X_{\varphi}$ in place of $\bR$. 

  \item\label{item:th:upperhemi:famxuhc} Analogue of \Cref{item:th:upperhemi:famrealuhc}, but with $\varphi\in \cF$ taking values in Hausdorff spaces $X_{\varphi}$. 
  \end{enumerate}
\end{theorem}
\begin{proof}
  \begin{description}[wide=0pt]
  \item {\bf \Cref{item:th:upperhemi:g0inv} $\xRightarrow{\quad}$ \Cref{item:th:upperhemi:g0rl}:} As noted in \cite[footnote 8]{zbMATH03072119}, this is a consequence of \cite[Theorem 3]{zbMATH03067716}.

  \item {\bf \Cref{item:th:upperhemi:g0rl} $\xRightarrow{\quad}$ \Cref{item:th:upperhemi:openrk}:} Specifically, we prove the formally stronger, Lie branch of \Cref{item:th:upperhemi:openrk}. 

    By \cite[Lemma 2.3.1]{mz} again, $\bG$ has an open almost connected subgroup $\bH\le \bG$. This will be the desired $\bH$: the claim, in particular, is that {\it any} almost connected $\bH$ will do.

    The celebrated Lie-group-approximation theorem for almost connected locally compact groups \cite[\S 4.6, Theorem]{mz} says that $\bH$ (being almost connected) is the cofiltered limit of its Lie quotients (i.e. is {\it pro-Lie} in any of the alternative senses of the term \cite[Theorem 3.39 and Definitions A, B and C preceding it]{hm_pro-lie-bk}). Note also that the identity component $\bG_0=\bH_0\cong \bR^d\times \bL$ {\it surjects} onto each identity component $\bH_{i,0}$, for instance because images of profinite groups (such as $\bH/\bH_0$) are again profinite \cite[Exercise E1.13]{hm4}.

    For an index $i$ large enough to ensure that the $\bR^d$ factor in $\bH_0\cong \bR^d\times \bL$ embeds into $\bH_{i,0}$, that identity component will be both in $\bH_i$ (because the latter is Lie) and of the desired form $\bR^d\times \bL_i$.
    
  \item {\bf \Cref{item:th:upperhemi:openrk} $\xRightarrow{\quad}$ \Cref{item:th:upperhemi:famrealconst}:} This time we hypothesize the formally weaker, non-Lie version of \Cref{item:th:upperhemi:openrk}.

    We can first declare all functions $\varphi\in \cF$ constant and non-zero on all non-trivial cosets of the given open subgroup $\bH\le \bG$, so that they are also obviously constant on the $\bG_0$-conjugacy classes outside $\bH$.

    Next, the restrictions $\varphi|_{\bH}$ will be pullbacks along the surjections $\bH\xrightarrowdbl{}\bH_i$ of functions defined on $\bH_i$, so it suffices to turn attention to an individual $\bH_i$. We will make implicit, repeated use of the fact that the identity component $\bH_0=\bG_0$ surjects onto $\bH_{i,0}$. 

    By assumption, $\bH_i$ has an open subgroup of the form $\bR^d\times \bL_i$ with compact $\bL_i$. The factor $\bR^d$ being central in
    \begin{equation*}
      \bH_{i,0}\cong \bR^d\times \bL_{i,0},
    \end{equation*}
    the conjugation action by $\bG_0$ factors through that of the {\it compact} group $\bL_{i,0}$. But then the quotient
    \begin{equation}\label{eq:quotbyconj}
      \bH_i/\left(\text{$\bG_0$-conjugation}\right)
      =
      \bH_i/\left(\text{$\bL_{i,0}$-conjugation}\right)
    \end{equation}
    (with its quotient topology) is locally compact Hausdorff \cite[\S III.4.1, Proposition 2 and its Corollary 1 and \S III.4.2, Proposition 3]{bourb_top_en_1}, hence \cite[Theorem 19.3]{wil-top} also {\it Tychonoff} \cite[Definition 14.8]{wil-top}: points can be separated from closed sets by continuous functions. We can now take the restrictions $\varphi|_{\bH}$ of $\varphi\in \cF$ to be the continuous real-valued functions on \Cref{eq:quotbyconj} with varying $i$, vanishing at (the image of) $1$, precomposed with the quotient maps $\bH\xrightarrowdbl{}\bH_i$.
    
  \item {\bf \Cref{item:th:upperhemi:famrealconst} $\xRightarrow{\quad}$ \Cref{item:th:upperhemi:famrealuhc}, \Cref{item:th:upperhemi:famxconst}, \Cref{item:th:upperhemi:famxuhc}} are all formal. One last implication will complete the picture.

  \item {\bf \Cref{item:th:upperhemi:famxuhc} $\xRightarrow{\quad}$ \Cref{item:th:upperhemi:g0inv}:} This is very much parallel to the proof of \cite[Lemma 1]{zbMATH03072119}.
    
    Fix a compact neighborhood $K\ni 1$ in $\bG$. Having chosen, for each point $p\in \partial K$ on the (compact) boundary of $K$, a function $\varphi_p\in \cF$ and disjoint neighborhoods
    \begin{equation*}
      V_{p,1}\ni \varphi_p(1)
      ,\quad
      V_{p,p}\ni \varphi_p(p)
      \quad\text{in }X_{\varphi_p}
    \end{equation*}
    (possible by the separation and Hausdorff assumptions), there is a {\it finite} cover of $\partial K$ with neighborhoods
    \begin{equation*}
      U_i:=\varphi_{p_i}^{-1}\left(V_{p_i,p_i}\right)\ni p_i
      ,\quad
      1\le i\le n.
    \end{equation*}
    The upper hemicontinuity assumption then ensures the existence of a neighborhood $U\ni 1$ in $\bG$ such that
    \begin{equation*}
      p\in U
      \xRightarrow{\quad}
      \varphi_{p_i}\left(\text{$\bG_0$-conjugacy class of $p$}\right)\subseteq V_{p_i,1}
      ,\quad \forall 1\le i\le n.
    \end{equation*}
    Because $V_{p_i,1}$ and $V_{p_i,p_i}$ are disjoint, this means that the $\bG_0$-conjugacy classes of $p\in U$ cannot intersect any of the $U_i$ and hence $\partial K$. But said conjugacy classes are connected, meaning that the union of the $\bG_0$-conjugates of $U$ is (a relatively compact neighborhood of $1$) contained in the prescribed $K$. 
  \end{description}
  This settles the mutual equivalence of all of the listed conditions. 
\end{proof}

One reason for the verbosity of \Cref{item:th:upperhemi:openrk} is that in general, when $\bG$ is not Lie, the mutually-equivalent conditions of \Cref{th:upperhemi} do not imply the existence of an open subgroup of the form $\bR^d\times \bL$ (even if $\bG$ is assumed almost connected).

\begin{example}\label{ex:noopen}
  Consider the group
  \begin{equation*}
    \bG := \bG_0\rtimes \left(\bZ/2\right)^{\aleph_0}
    ,\quad
    \bG_0 := \bR\times \bT
    ,\quad
    \bT := (\bR/\bZ)^{\aleph_0},
  \end{equation*}
  with the action defined as follows:
  \begin{itemize}
  \item the generator $\sigma_n$ of the $n$-indexed factor $\bZ/2$ acts by inversion on the $n$-indexed circle factor $(\bR/\bZ)_n$ and trivially on $(\bR/\bZ)_m$, $m\ne n$;
  \item and the same $\sigma_n$ operates on the factor $\bR\cong \bR\times \{0\}\subset \bG_0$ by
    \begin{equation*}
      \bR
      \xrightarrow[(\id,\ \text{obvious surjection onto $n^th$ component})]{\quad\sigma_n\quad}
      \bR\times (\bR/\bZ)^{\aleph_0}.
    \end{equation*}        
  \end{itemize}
  The centralizer in $\bG$ of the factor $\bR\subset \bG_0$ is precisely $\bG_0$, so in particular has infinite index in $\bG$. 
\end{example}

\begin{remark}\label{re:map}
  \Cref{ex:noopen} is very much in the spirit of \cite[Part II]{zbMATH03060242}, and can serve the same purpose: the group $\bG$ therein is {\it maximal almost periodic} (or {\it MAP} \cite[\S 1]{zbMATH03268228}: it has sufficiently many finite-dimensional unitary representations), but its {\it left and right uniform structures} \cite[\S 7.1, post Definition 7.3]{james_unif_1999} do not coincide. 
  
  Indeed, condition \Cref{item:th:upperhemi:g0rl} in \Cref{th:upperhemi} means precisely \cite[\S 32]{zbMATH03108006} that $\bG_0$ is MAP. The group $\bG$ of \Cref{ex:noopen} being almost connected, it is MAP (\cite[Theorem 2]{zbMATH03060242} or \cite[Corollary 3.1]{MR0529213}) along with its identity component $\bG_0$. 
  
  On the other hand, locally compact groups whose left and right uniformities coincide contain \cite[\S 4, Proposition]{zbMATH03060242} normal open subgroups of the form $\bR^d\times \bL$ with $\bL$ compact; the group $\bG$ of \Cref{ex:noopen} does not. 
\end{remark}

As the discussion preceding \Cref{th:upperhemi} suggests, all of this affords generalizations of various results on the structure of unitary  uniformly continuous representations on Banach spaces by connected locally compact groups. \Cref{th:enoughreps}, for instance, is meant to extend \cite[Corollary 4]{zbMATH03072119}, where the locally compact group is assumed connected and the representation unitary. Some prefatory remarks are in order. 

\begin{conventions}\label{cvs:lctvs}
  We frequently work with locally convex TVSs, which phrase we occasionally shorten to {\it LCTVS}. $\cL(E,F)$ denotes the space of continuous linear maps $E\to F$, with $\cL(E,E)$ abbreviated to $\cL(E)$. 

  Following \cite[\S 39.1]{k_tvs-2} or \cite[\S III.3.1]{bourb_tvs} or \cite[\S X.1]{bourb_top_en_2} in the broader context of uniform spaces, we write $\cL_{\fM}(E,F)$ for $\cL(E,F)$ equipped with the topology of uniform convergence on each member of a family $\fM$ of subsets of $E$ (the {\it $\fM$-topology}-topology, or {\it topology of $\fM$-convergence}).  

  We assume $\fM$ to be a {\it covering bornology} \cite[\S III.1.1, Definition 1]{bourb_tvs}: closed under taking subsets and finite unions, and with $\bigcup_{\fM}B=E$.  The $\fM$ subscripts apply also to the spaces $GL(E)$ of (topological and linear) automorphisms of $E$, as in $GL_{\fM}(E)$. The members of $\fM$ are often (but not necessarily) {\it bounded} \cite[\S 15.6]{k_tvs-1}. $\cL_{\fM}(E,F)$ is
  \begin{itemize}
  \item Hausdorff if $F$ is \cite[\S X.1.2, Proposition 1]{bourb_top_en_2} (recall that we are assuming $\fM$ covers $E$);

  \item a TVS and locally convex if $F$ is \cite[\S III.3.1, Proposition 1]{bourb_tvs}, provided linear continuous maps $E\to F$ send members of $\fM$ to bounded sets. In particular, this is the case \cite[Proposition 14.2]{trev_tvs} provided $\fM$ itself is bounded (i.e. consists of bounded sets). 
  \end{itemize}
\end{conventions}

\begin{definition}\label{def:rep}
  Let $\bG$ be a topological group and $E$ a TVS (typically but not necessarily locally convex).

  \begin{enumerate}[(1)]
  \item For a family $\fM$ as in \Cref{cvs:lctvs} an {\it $\fM$-representation} of $\bG$ on $E$ (or an $\fM$-$\bG$-representation) is a continuous morphism $\bG\to GL_{\fM}(E)$.

  \item Plain (unqualified) representations are $\fM$-representations
    \begin{equation}\label{eq:g2gls}
      \bG
      \xrightarrow{\quad}
      GL_{s}(E)
      :=
      GL_{\fM}(E)
      ,\quad
      \fM:=\left\{\text{all finite subsets}\right\}.
    \end{equation}
    The `s' subscript stands for the {\it simple} \cite[post \S 39.1(4)]{k_tvs-2} (or pointwise) topology.

    Representations are thus {\it separately continuous} \cite[\S VIII.2.1, Definition 1 (i)]{bourb_int_en_7-9}, in the sense that the induced action
    \begin{equation}\label{eq:gone}
      \bG\times E\to E
    \end{equation}
    is continuous in each variable when the other is fixed (as customary: \cite[Definition 2.1 and Remark 2.2]{hm4} or \cite[Definition 13.1.1]{dixc} or \cite[\S 2]{rob}). We thus also need

  \item {\it Jointly continuous} (or just {\it joint}) $\fM$-representations are \cite[\S VIII.2.1, Definition 1 (ii)]{bourb_int_en_7-9} those for which  \Cref{eq:gone} is continuous for the product topology on the domain.     
  \end{enumerate}
\end{definition}

\begin{theorem}\label{th:enoughreps}
  For a locally compact group $\bG$ the following conditions are equivalent.
  \begin{enumerate}[(a)]
  \item\label{item:th:enoughreps:oldconds} $\bG_0\cong \bR^d\times \bK$ for some compact connected $\bK$ and $d\in \bZ_{\ge 0}$.

  \item\label{item:th:enoughreps:hilb} $\bG$ has a {\it separating} family of unitary, norm-continuous representations, in the sense that
    \begin{equation*}
      \bG\ni g\ne 1
      \xRightarrow{\quad}
      \varphi(g)\ne 1\text{ for some }\varphi\in \cF.
    \end{equation*}
    
  \item\label{item:th:enoughreps:ban} $\bG$ has a separating family of norm-continuous representations $\bG\xrightarrow{\varphi\in \cF}GL(E)$ on Banach spaces $E=E_{\varphi}$, all bounded in the sense that $\sup_{g\in \bG}\left\|\rho(g)\right\|<\infty$ for all $\varphi\in \cF$.    
    
  \item\label{item:th:enoughreps:lctvs} $\bG$ has a separating family of $\fM$-representations $\bG\xrightarrow{\varphi\in \cF}GL_{\fM}(E)$ on topological vector spaces $E=E_{\varphi}$ such that
    \begin{itemize}
    \item for each $\varphi\in \cF$, the family $\varphi(\bG_0)\subset GL_{\fM}(E)$ is {\it equicontinuous} \cite[\S 15.13]{k_tvs-1} (equivalently \cite[\S 15.13(1)]{k_tvs-1}, {\it equicontinuous at $0\in E$}):

      \begin{equation}\label{eq:equicontat0}
        \forall\text{ neighborhood }U\ni 0,\quad\exists\text{ neighborhood }V\ni 0 \text{ with }\varphi(\bG_0)V\subseteq U.
      \end{equation}
      
    \item and every $\varphi|_{\bG_0}$ leaves $\fM$ invariant:
      \begin{equation}\label{eq:fmg0inv}
        \varphi(\bG_0)M\in \fM
        ,\quad
        \forall M\in \fM
        ,\quad
        \forall \varphi\in \cF. 
      \end{equation}
    \end{itemize}
  \end{enumerate}  
\end{theorem}
\begin{proof}
  \begin{description}[wide=0pt]

  \item {\bf \Cref{item:th:enoughreps:oldconds} $\xRightarrow{\quad}$ \Cref{item:th:enoughreps:hilb}:} The hypothesis makes available all of the various mutually-equivalent conditions of \Cref{th:upperhemi}, and in particular \Cref{item:th:upperhemi:openrk}. Observe first that it will be enough to work with the open subgroup $\bH\le \bG$ provided by that condition: assuming we have such representations $\bH\xrightarrow{\varphi\in \cF}U(E_{\varphi})$ (unitary groups of the Hilbert spaces $E_{\varphi}$), we can {\it induce} them up to representations
    \begin{equation*}
      \bG
      \xrightarrow{\quad\Ind_{\bH}^{\bG}\varphi\quad}
      GL(F_{\varphi})
      \quad\text{for appropriate }F_{\varphi}.
    \end{equation*}
    The openness of $\bH$ will obviate many of the complications \cite[\S E.1]{bdv} attendant to induction from arbitrary closed subgroups, so we can follow \cite[\S 2.1]{kt}. Having fixed a unitary representation $\varphi\in \cF$ on $E=E_{\varphi}$, set
    \begin{equation*}
      F=F_{\varphi}:=
      \left\{
        \bG\xrightarrow{f}E
        \ |\
        f(gh) = \varphi(h^{-1})f(g),\ \forall g\in G,\ h\in H
        \quad\text{and}\quad
        \sum_{x\in \bG/\bH}\|f(\widetilde{x})\|^2<\infty
      \right\},
    \end{equation*}
    where $\widetilde{x}\in \bG$ are arbitrary lifts of $x\in \bG$. The lifts make no difference to $\|f(\widetilde{x})\|$ because $\varphi$ is unitary. The action of $\bG$ on $F$ is (as expected)
    \begin{equation*}
      \Ind_{\bH}^{\bG}\varphi(g)f := f(g^{-1}\bullet),
    \end{equation*}
    and verifying that the desired properties transport over from the family $\{\varphi\}$ to its induced counterpart is a simple matter.
    
    Having reduced the problem to the $\bH$ of \Cref{th:upperhemi}\Cref{item:th:upperhemi:openrk}, we can further pass on to the quotients $\bH_i$ (since we can regard $\bH_i$-representations as $\bH$-representations), and then reduce again to the open subgroups $\bR^d\times \bL_i\le \bH_i$.
    
    All in all, it is enough to focus on groups of the form $\bR^d\times \bK$, for which the claim is obvious: simply take the $\varphi$ to be external tensor products of irreducible unitary $\bR^d$- and $\bK$-representations, both of which will be finite-dimensional (and hence norm-continuous). Or: this is the easy direction of the characterization \cite[\S 32]{zbMATH03108006} of connected locally compact maximally almost periodic groups, and goes through verbatim in the disconnected case as well. 
    
  \item {\bf \Cref{item:th:enoughreps:hilb} $\xRightarrow{\quad}$ \Cref{item:th:enoughreps:ban}} is immediate.

  \item {\bf \Cref{item:th:enoughreps:ban} $\xRightarrow{\quad}$ \Cref{item:th:enoughreps:lctvs}:} The family of \Cref{item:th:enoughreps:ban} satisfies the conditions in \Cref{item:th:enoughreps:lctvs}. Indeed, norm-continuity means continuity for the {\it strong} \cite[post \S 39.1(4)]{k_tvs-2} topology
    \begin{equation*}
      GL_b(E)
      :=
      GL_{\fM}(E)
      ,\quad
      \fM:=\left\{\text{all bounded subsets}\right\}.
    \end{equation*}
    For Banach (or indeed, complete metrizable \cite[\S 15.13(2)]{k_tvs-1}) spaces $E$ equicontinuity is precisely uniform boundedness, which then also implies \Cref{eq:fmg0inv}.

  \item {\bf \Cref{item:th:enoughreps:lctvs} $\xRightarrow{\quad}$ \Cref{item:th:enoughreps:oldconds}:} It will be most convenient to check condition \Cref{item:th:upperhemi:famxuhc} of \Cref{th:upperhemi}. Indeed, simply take the requisite $\varphi$ to be the representations provided by the hypothesis, i.e.
    \begin{equation*}
      \bG\xrightarrow{\quad\varphi\quad}\cL_{\fM}(E_{\varphi})
      ,\quad
      \varphi\in \cF.
    \end{equation*}
    The separation requirement is built into the hypothesis, so the one (mildly) sticking point is upper hemicontinuity. Verifying it means showing that for fixed $\varphi\in \cF$ operating on $E:=E_{\varphi}$
    \begin{equation*}
      \bG\ni g
      \text{ close to }1
      \quad
      \xRightarrow{\quad}
      \quad
      \varphi\left(Ad_{\bG_0}(g)\right)\subset\text{ small neighborhood }U\ni 1\in \cL_{\fM}(E),
    \end{equation*}
    where
    \begin{equation*}
      Ad_{\bG_0}(g):=\left\{sgs^{-1}\ |\ s\in \bG_0\right\}
    \end{equation*}
    is the $\bG_0$-conjugacy class of $g$. Being contained in $U$, in turn, means that all operators in $\varphi\left(Ad_{\bG_0}(g)\right)$ are uniformly close to the identity on a fixed bounded subset $M\in \fM$. We write `$\sim$' for `is close to'. The conclusion follows from the two constraints imposed on the individual $\varphi$:
    \begin{itemize}
    \item the set $\varphi(\bG_0)M$ still belongs to $\fM$ by \Cref{eq:fmg0inv};

    \item so that $g\in \bG$ close to $1$ are uniformly close to the identity on that set:
      \begin{equation*}
        \varphi(g)\varphi(s^{-1})
        \sim
        \varphi(s^{-1})
        \text{ uniformly on $M$ over $g\sim 1$};
      \end{equation*}

    \item whence
      \begin{equation*}
        \varphi(s)\varphi(g)\varphi(s^{-1})
        \sim
        \varphi(s)\varphi(s^{-1})
        =1
        \text{ uniformly on $M$ over $g\sim 1$}
      \end{equation*}
      by the equicontinuity assumption. 
    \end{itemize}
  \end{description}
\end{proof}

\begin{remark}\label{re:equicont-bdd}
  The two conditions \Cref{eq:equicontat0} and \Cref{eq:fmg0inv} of \Cref{th:enoughreps}\Cref{item:th:enoughreps:lctvs} are not exactly unrelated. For one thing, assuming
  \begin{equation}\label{eq:mallbdd}
    \fM:=\left\{\text{all bounded sets}\right\},
  \end{equation}
  the equicontinuity requirement \Cref{eq:equicontat0} plainly implies the boundedness condition \Cref{eq:fmg0inv}: having fixed bounded $B\subseteq E$ and a neighborhood $U\ni 0\in E$,
  \begin{equation*}
    \begin{aligned}
      U
      &\supseteq
      \varphi(\bG_0)V
      \quad\text{for some origin neighborhood $V$ by \Cref{eq:equicontat0}}\\
      &\supseteq
        \lambda \varphi(\bG_0)B
        \quad\text{for small scalars $\lambda$ because $B\subseteq \frac{1}{\lambda}V$ is bounded},\\
    \end{aligned}    
  \end{equation*}
 and hence $\varphi(\bG_0)B$ is bounded. 

 If in addition (to \Cref{eq:mallbdd}) $E$ is {\it locally bounded} (i.e. has a bounded origin neighborhood; equivalently \cite[\S 15.10]{k_tvs-1}: $E$ is {\it quasi-normable} or {\it $p$-normable} for some $0<p\le 1$) then \Cref{eq:equicontat0,eq:fmg0inv} are easily proven equivalent. In particular, this applies to normable spaces (= locally bounded + locally convex \cite[\S 15.10(4)]{k_tvs-1}). Nevertheless, in the generality of \Cref{th:enoughreps}\Cref{item:th:enoughreps:lctvs} the two conditions are independent: \Cref{ex:equicont-not-bdd,ex:bdd-not-equicont} show that each can in principle occur without the other.
\end{remark}

\begin{example}\label{ex:equicont-not-bdd}
  Assuming a compact group $\bG$ acts on $E$ {\it jointly} continuously via $\bG\xrightarrow{\varphi}GL_{\fM}(E)$, the family $\varphi(\bG)$ is easily seen to be equicontinuous \cite[\S VIII.2.1, Remark 2) ]{bourb_tvs}. On the other hand, if, say, $\fM$ is the family of finite sets and $\bG_0$ has at least one infinite orbit, then of course $\fM$ will not be preserved by the action in the sense of \Cref{eq:fmg0inv}. 
\end{example}

\begin{example}\label{ex:bdd-not-equicont}
  Assuming \Cref{eq:fmg0inv} even for the family $\fM$ of {\it all} bounded sets, \Cref{eq:equicontat0} need not follow. The subsequent construction will do double duty in \Cref{ex:hm-joint-cont:pro-cast-not-jointly-cont}. We work over the complex numbers, so as to fall within the scope of \cite{phil_inv-lim}. Specifically, we adapt \cite[Example 2.11]{phil_inv-lim} to the present purposes.

  Take $\bG:=\bS^1$, say, and let $E$ be the {\it pro-$C^*$-algebra} \cite[Definition 1.1]{phil_inv-lim} of functions on $\bG$ defined by analogy to the aforementioned example (where the base space is $[0,1]$ instead of the circle):
  \begin{equation*}
    E:=\cC_{\bC}(\bG):=\left\{\text{continuous functions }\bG\to \bC\right\},
  \end{equation*}
  with the topology of uniform convergence on every countable compact subset $F\subset \bG$ with finitely many cluster points (call such $F\subset \bG$ {\it distinguished}, per \cite[Definition 2.5]{phil_inv-lim}). The topology on $E$ is complete and locally convex (by general theory \cite[\S 19.8(1) and \S 19.10(2)]{k_tvs-1}: $E$ is a cofiltered limit of $C^*$-algebras $\cC_{\bC}(F)$). $\bG$ acts on $E$ by translation on the common domain $\bG$ of the functions.

  The bounded subsets $B\subseteq E$ are precisely those bounded in the usual sense:
  \begin{equation*}
    \sup_{f\in B,\ x\in \bG}|f(x)|<\infty. 
  \end{equation*}
  Indeed, abstract boundedness means \cite[Proposition 14.5]{trev_tvs} uniform boundedness on every distinguished $F\subseteq \bG$, and a sequence $x_n\in \bG=\bS^1$ with
  \begin{equation*}
    |f_n(x_n)|\xrightarrow[n]{\quad}\infty,\quad f_n\in B
  \end{equation*}
  would entail unboundedness on the distinguished $F$ consisting of a convergent subsequence of $(x_n)$ together with its limit. It follows from this that \Cref{eq:fmg0inv} holds: naturally, translations of sets of functions uniformly bounded on $\bG$ are again such. \Cref{eq:equicontat0} does not hold though: no matter how large a distinguished $F\subset \bG$ is, there are \cite[Theorems 15.8 and 17.10]{wil-top} $f\in E$ vanishing on $F$ and taking the value $1$ at some fixed $x\in \bG\setminus F$. 
\end{example}

Recall (e.g. \cite[\S 0]{zbMATH03680105}) that the {\it von Neumann kernel} $\fn(\bG)\trianglelefteq \bG$ of a locally compact group $\bG$ is the intersection of the kernels of finite-dimensional unitary $\bG$-representations; it is a closed characteristic subgroup, as observed in loc.cit. \cite[Lemmas 1 and 2]{klm_unif} can be rephrased as saying that for connected locally compact $\bG$, $\fn(\bG)$ can also be recovered as the intersection of the kernels of all norm-continuous unitary representations (on perhaps infinite-dimensional Hilbert spaces). The following immediate consequence of \Cref{th:enoughreps} amplifies that result. 

\begin{corollary}\label{cor:vnker-tvs}
  For a locally compact group $\bG$, the von Neumann kernel $\fn(\bG_0)$ of its identity component is precisely the joint kernel of all 
  \begin{enumerate}[(a)]
  \item norm-continuous unitary $\bG$-representations;
  \item or norm-continuous, bounded representations on Banach spaces;
  \item or representations $\bG\xrightarrow{\varphi} GL_{\fM}(E)$ on topological vector spaces with $\varphi|_{\bG_0}$ equicontinuous and $\fM$-preserving in the sense of \Cref{eq:equicontat0} and \Cref{eq:fmg0inv} respectively.  \qedhere
  \end{enumerate}
\end{corollary}

%\newpage

%%%%%%%%%%%%%%%%%%%%%%%%%%%%%%%%%%%%%%%%%%%%%%%%%%%%%%%%%%%%%%%%%%%%%%%%%%%%%
%%%%%%%%%%%%%%%%%%%%%%%%%%%%%%%%%%%%%%%%%%%%%%%%%%%%%%%%%%%%%%%%%%%%%%%%%%%%%
\section{Locally and globally finite representations}\label{se:hm4-e48}

Recall \cite[Exercise E4.8]{hm4}: a representation of a compact group $\bG$ on a (real) {\it weakly complete} \cite[Definition A7.8]{hm4} TVS $E$ (i.e. a topological product of real lines) for which the space
\begin{equation}\label{eq:efin}
  E_{fin}:=\left\{v\in E\ |\ \dim \spn \bG v<\infty\right\}
\end{equation}
of {\it $\bG$-finite} \cite[Definition 3.1]{hm4} vectors is all of $E$ has finitely many {\it isotypic components} \cite[Definition 4.21]{hm4}. Taking a cue from that source, the present section's aim is to analyze that automatic-finiteness phenomenon in more depth.

\begin{remarks}\label{res:hm-joint-cont}
  \begin{enumerate}[(1),wide=0pt]
    
  \item\label{item:res:hm-joint-cont:wc-barreled} As noted in \Cref{cvs:lctvs}, actions $\bG\times E\to E$ attached to representations of topological groups $\bG$ on TVSs $E$ are assumed {\it separately} continuous in \cite[Definition 2.1]{hm4}. On the other hand, the hint provided for \cite[Exercise E4.8]{hm4} appeals to \cite[Theorem 4.22]{hm4}, valid for {\it jointly} continuous representations (and in general {\it not} without that hypothesis: \Cref{ex:hm-joint-cont:pro-cast-not-jointly-cont}).
  
    The reason why all is well is that being products of barreled spaces, weakly complete TVSs are again barreled \cite[\S 27.1(5)]{k_tvs-2} (indeed, Montel \cite[\S 27.2(4)]{k_tvs-1}), and hence separately continuous representations thereon of locally compact groups are automatically jointly continuous by \cite[\S VIII.2.1, Proposition 1]{bourb_int_en_7-9} or \cite[Proposition 13.2]{rob}.
    
  \item\label{item:res:hm-joint-cont:wc-confusion} The term {\it weakly complete} recalled in \Cref{item:res:hm-joint-cont:wc-barreled} above is liable to cause some confusion: it is obviously used differently in \cite[\S 10.6(6)]{k_tvs-1}, where infinite-dimensional spaces are {\it never} weakly complete (in that text's sense). The difference lies in this: the weakly complete spaces of \cite[Lemma A7.7 and Definition A7.8]{hm4} are cofiltered limits of finite-dimensional real or complex vector spaces with their {\it standard} topologies, whereas in \cite[\S 10.4(8)]{k_tvs-1} the topologies in question have closed cofinite subspaces as a basis of neighborhoods for the origin, so that the resulting finite-dimensional quotients are {\it discrete}.

    In short: the difference between \cite[Appendix 7]{hm4} and \cite[\S 10]{k_tvs-1} is that between equipping the ground field ($\bR$ or $\bC$ in the cases of interest here) with the standard or the discrete topology respectively. 

  \item\label{item:res:hm-joint-cont:pro-cast-not-jointly-cont} In reference to \Cref{item:res:hm-joint-cont:wc-barreled} above, it might be worth pointing out to what extent joint continuity is necessary in the decomposition \cite[Theorem 4.22]{hm4}, for compact $\bG$, of a $\bG$-representation as a completed direct sum of isotypic components. Unwinding that argument back to \cite[Theorem 3.36(vi)]{hm4}, it becomes apparent that joint continuity is appealed to in arguing that the idempotent {\it averaging operator} \cite[Definition 3.32]{hm4}
    \begin{equation}\label{eq:averageop}
      E\ni v
      \xmapsto{\quad}
      \int_{\bG}gv\ \mathrm{d}\mu_{Haar}(g)
      \in E,
    \end{equation}
    picking out the $\bG$-invariant elements of $E$, is continuous. \Cref{ex:hm-joint-cont:pro-cast-not-jointly-cont} shows that this is no longer the case absent {\it joint} continuity. 
  \end{enumerate}
\end{remarks}

\begin{example}\label{ex:hm-joint-cont:pro-cast-not-jointly-cont} \Cref{ex:bdd-not-equicont} will do for this purpose as well: $E$ consists of
  \begin{equation*}
    \bG:=\bS^1\xrightarrow[\quad\text{continuous}\quad]{f}\bC
  \end{equation*}
  with the topology of uniform convergence on every distinguished (meaning countable with finitely many cluster points) closed $F\subset \bG$, and $\bG$ acts by translation. The separate continuity of the action is checked immediately, whereas {\it joint} continuity decidedly fails. This can be verified
  \begin{itemize}
  \item directly, by dualizing \cite[Theorem 2.7]{phil_inv-lim} the action $\bG\times E\to E$ to the translation action $\bG\times \bG\to \bG$ and observing that the latter does not preserve the family of distinguished compact sets $F\subset \bG$ mentioned above;

  \item or, more to the point and downstream from \cite[Theorem 3.36(vi)]{hm4}, by noting that the averaging operator \Cref{eq:averageop} is not continuous because its kernel
    \begin{equation*}
      \left\{f\in \cC_{\bC}(\bG)\ \bigg|\ \int_{\bG}f\ \mathrm{d}\mu_{Haar} = 0\right\}
    \end{equation*}
    is not closed in the topology used here. 
  \end{itemize}
  Indeed, said kernel is dense in $E$: for every continuous $\bG\xrightarrow{f}\bC$ and every distinguished $F\subset \bG$ one can alter $f$ outside $F$ so as to obtain a function $f'$ with vanishing integral but with $f'|_F\equiv f|_F$.
\end{example}

\begin{remark}\label{re:notbarreled}
  Naturally, by \Cref{res:hm-joint-cont}\Cref{item:res:hm-joint-cont:wc-barreled}, the space $E$ of \Cref{ex:hm-joint-cont:pro-cast-not-jointly-cont} is not barreled. This also follows directly from \cite[Remark 4.1]{inoue_loc}, giving a necessary condition for a commutative pro-$C^*$-algebra to be barreled
\end{remark}

We dwell (via \Cref{ex:actonpolys}) on the issue of joint versus separate continuity for actions for a moment, because it acts as a kind of motivating nexus for a number of other problems one is naturally led to consider in the present context: the (non-)extensibility of an action on a (locally convex) TVS to its {\it completion} \cite[\S 15.3]{k_tvs-1}, the interplay between being barreled and normed (\Cref{res:manymore}\Cref{item:res:manymore:normed-not-barr}), the strict gradation of completeness properties \cite[Table 3.1]{hm4} pertinent to actions, and so on.

\begin{example}\label{ex:actonpolys}
  This is again a separately but not jointly continuous action, where the main ingredient is a completeness deficiency.
  
  Consider the rotation action of $\bG:=\bS^1\subset \bC$ on the space $E$ of (complex, say) polynomials on $\bC$, topologized with the supremum norm over a proper compact arc $A\subset \bS^1$. Non-zero polynomials being non-zero on $A$, this is a norm and the space $E$ is normed but not complete. The rotation action is also easily checked to indeed be a representation, i.e. separately continuous as a map $\bG\times E\to E$.

  The completion $\overline{E}$ is the space $\cC_{\bC}(A)$ of complex-valued continuous functions on $A$ with its usual supremum norm $\sup_A|\cdot|$, and the rotation action on $E$ obviously does not extend (even {\it separately} continuously) to $\overline{E}$. The original action, then, could not have been {\it jointly} continuous: all such actions extend to completions automatically (\Cref{le:ext2compl}).  
\end{example}

\begin{remarks}\label{res:manymore}
  \begin{enumerate}[(1),wide=0pt]
  \item\label{item:res:manymore:gc} The reader can easily extrapolate from \Cref{ex:actonpolys} to much broader classes. One could, for instance, substitute
    \begin{itemize}
    \item any compact Lie group $\bG$ for $\bS^1$, embedded in its associated complex algebraic group $\bG_{\bC}$ (the embedding $\fG\subset \fM(\fG)$ of \cite[Chapter VI, \S VIII, Definition 1]{chev_lie-bk_1946});
    \item any compact $A\subset \bG_{\bC}$ with non-empty interior for the arc $A\subset \bS^1$;

    \item and any number of algebras of ``sufficiently regular'' functions on $\bG_{\bC}$ for the polynomials of \Cref{ex:actonpolys}: the holomorphic functions on the complex manifold $\bG_{\bC}$ are one option, and the {\it representative} \cite[Definition 3.3]{hm4} functions (elements of the {\it representative ring} of $\bG$ in \cite[Chapter VI, \S VII]{chev_lie-bk_1946}) are another. 
    \end{itemize}

  \item\label{item:res:manymore:normed-not-barr} Being normed, the space $E$ of \Cref{ex:actonpolys} is by \cite[\S 28.1(4)]{k_tvs-1} also {\it bornological} \cite[\S 28.1]{k_tvs-1}: absolutely convex sets which {\it absorb} \cite[post \S 23.4(2)]{k_tvs-1} arbitrary bounded sets are origin neighborhoods. It cannot be barreled, for then the separately continuous action would automatically be jointly so (\cite[\S VIII.2.1, Proposition 1]{bourb_int_en_7-9} again).

    \Cref{ex:actonpolys} and its analogues, then, provide natural instances of bornological (indeed, normed) spaces that are not barreled; cf. \cite[\S 21.5, concluding paragraphs]{k_tvs-1} for another such. 
  \end{enumerate}
\end{remarks}

The following (presumably well-known) result is included here for ease of access and reference. 

\begin{lemma}\label{le:ext2compl}
  Let 
  \begin{equation}\label{eq:gonx}
    \bG\times X
    \xrightarrow{\quad\triangleright\quad}
    X
  \end{equation}
  be a continuous action of a locally compact group on a Hausdorff uniform space, with each $X\xrightarrow{g\triangleright}X$ uniformly continuous. 

  The canonical extension $\bG\times \widehat{X}\xrightarrow{} \widehat{X}$ of $\triangleright$ to the completion of $X$ is again continuous.
\end{lemma}
\begin{proof}
  The functoriality \cite[\S II.3.7, Theorem 3]{bourb_top_en_1} of the construction $X\mapsto \widehat{X}$ does indeed provide a canonical extension as in the statement: every (uniformly continuous)
  \begin{equation*}    
    X\xrightarrow{g\triangleright}X
    \text{extends uniquely to }
    \widehat{X}\xrightarrow{g\triangleright}\widehat{X}
    \text{ (also uniformly continuous)},
  \end{equation*}
  and the associativity and unitality of the resulting action follow from said functoriality. The issue is the joint continuity of the extension $\bG\times \widehat{X}\to \widehat{X}$. First recast \cite[\S VII.8, equation (2)]{mcl_2e} the initial action \Cref{eq:gonx} as a map
  \begin{equation}\label{eq:x2g2x}
    X
    \xrightarrow{\quad}
    \cat{Cont}(\bG,X)
  \end{equation}
  uniformly continuous for the uniformity of compact convergence \cite[\S X.1.3, example III]{bourb_top_en_2} on the codomain. Said compact-convergence uniformity is complete \cite[\S X.1.6, Corollary 3 to Theorem 2]{bourb_top_en_2}, so \Cref{eq:x2g2x} fits into
  \begin{equation*}
    \begin{tikzpicture}[auto,baseline=(current  bounding  box.center)]
      \path[anchor=base] 
      (0,0) node (l) {$X$}
      +(2,.5) node (u) {$\widehat{X}$}
      +(2,-.5) node (d) {$\cat{Cont}(\bG,X)$}
      +(6,0) node (r) {$\cat{Cont}(\bG,\widehat{X})$.}
      ;
      \draw[->] (l) to[bend left=6] node[pos=.5,auto] {$\scriptstyle $} (u);
      \draw[->] (u) to[bend left=6] node[pos=.5,auto] {$\scriptstyle $} (r);
      \draw[->] (l) to[bend right=6] node[pos=.5,auto,swap] {$\scriptstyle $} (d);
      \draw[->] (d) to[bend right=6] node[pos=.5,auto,swap] {$\scriptstyle $} (r);
    \end{tikzpicture}
  \end{equation*}
  The upper right-hand map is the counterpart (via \cite[\S VII.8, (2)]{mcl_2e} again) of the extended action $\bG\times \widehat{X}\to \widehat{X}$. 
\end{proof}

\cite[Exercise E4.8]{hm4} suggests the natural problem of extending it to actions on broader classes of locally convex TVSs. It will first be convenient to introduce some handy terminology.

\begin{definition}\label{def:gfin}
  \begin{enumerate}[(1)]
  \item\label{item:def:gfin:locfin} A representation $\varphi$ of a locally compact group on a TVS $E$ is {\it locally $\bG$-finite} if the space $E_{fin}$ of \Cref{eq:efin} coincides with $E$.

  \item\label{item:def:gfin:globfin} For compact $\bG$, $\varphi$ is {\it (globally) $\bG$-finite} if it is locally $\bG$-finite and has finitely many isotypic components.

  \item\label{item:def:gfin:gglob} $E$ is {\it $\bG$-globalizing} (for compact $\bG$) if every jointly continuous $\bG$-representation on $E$ is $\bG$-finite as soon as it is locally so.

  \item\label{item:def:gfin:glob} $E$ is {\it globalizing} if it is $\bG$-globalizing for arbitrary compact $\bG$. 
  \end{enumerate}
\end{definition}

In this language, \cite[Exercise E4.8]{hm4} says that weakly complete TVSs are globalizing. It is not difficult to show (\Cref{pr:hme48-frechet}) that {\it Fr\'echet} or {\it (F)}-spaces (i.e. \cite[\S 18.2]{k_tvs-1} complete metrizable locally convex) also are. Among products of real (or complex) lines, only those with $\le \aleph_0$ factors are metrizable \cite[\S 18.3(6)]{k_tvs-1}; weakly complete spaces, then, mostly are not. On the other hand, arbitrary complete (even barreled) LCTVSs are certainly not all globalizing, by \Cref{ex:hme48-dirsum}. {\it Some} constraint on the carrier space is thus necessary for globalization to go through. 

\begin{example}\label{ex:hme48-dirsum}
  Fix an arbitrary (possibly infinite) family of irreducible (hence finite-dimensional) representations $\bG\xrightarrow{\varphi}GL(V_{\varphi})$ of a compact group, and equip the direct sum $V:=\bigoplus_{\varphi}V_{\varphi}$ with its {\it direct-sum locally convex topology} \cite[\S 18.5]{k_tvs-1} and the obvious direct-sum $\bG$-action.

  $V$ is complete and locally convex \cite[\S 18.5(3)]{k_tvs-1}, the action is plainly jointly continuous (since every vector has non-zero components in only finitely many $V_{\varphi}$) and by construction, every vector is $\bG$-finite. If the family $\{\varphi\}$ was infinite to begin with, though, then the resulting representation on $V$ of course has infinitely many isotypic components.

  Examples of this general type are by \cite[\S 28.4(1)]{k_tvs-1} even bornological. Being also complete, they are barreled \cite[\S 28.1(2)]{k_tvs-1}. 
\end{example}

\begin{proposition}\label{pr:hme48-frechet}
  Fr\'echet spaces are globalizing in the sense of \Cref{def:gfin}\Cref{item:def:gfin:glob}. 
\end{proposition}
\begin{proof}
  Naturally, only the local-to-global implication is interesting, so let $\bG\xrightarrow{\varphi}GL_s(E)$ (notation as in \Cref{eq:g2gls}) be a representation on an $(F)$-space, automatically jointly continuous \cite[\S VIII.2.1, Proposition 1]{bourb_int_en_7-9} because $(F)$-spaces are barreled \cite[\S 21.5(3)]{k_tvs-1}.
  
  The topology of $E$ is by assumption induced by a complete metric $d$. Were there infinitely many isotypic components $E_n\le E$, $n\in \bZ_{\ge 0}$, we could pick
  \begin{equation*}
    v_n\in E_n,\quad 0<\|v_n\|:=d(0,v_n)<\frac {1}{2^n}.
  \end{equation*}
  The sequence of partial sums $v_0+\cdots + v_n$ is Cauchy and hence convergent to $v:=\sum_n v_n$, and the (continuous \cite[Theorem 4.22]{hm4}) projections $E\xrightarrowdbl{} E_n$ all fail to annihilate the limit $v$. This contradicts the local $\bG$-finiteness, and we are done.
\end{proof}

Another simple observation:

\begin{proposition}\label{pr:equivlim}
  Let $\bG$ be a compact group and $\bG\xrightarrow{\varphi}GL_s(E)$ a jointly continuous representation on a complete LCTVS expressible $\bG$-equivariantly as a cofiltered limit
  \begin{equation}\label{eq:gequivelimei}
    E\cong \varprojlim_{i}E_i
    ,\quad
    E_i:=E/H_i\text{ complete}
    ,\quad
    H_i\le E\text{ closed $\bG$-invariant}.
  \end{equation}
  If the induced $\bG$-representations on $E_i$ are all $\bG$-finite, then
  \begin{equation*}
    E\cong \prod_{\alpha\in\Irr\bG}E_{\alpha}
    \quad\left(\text{product of its isotypic components}\right),
  \end{equation*}
  and
  \begin{equation*}
    E_{\alpha}\ne 0
    \xLeftrightarrow{\quad}
    E_{i,\alpha}\ne 0
    \text{ for some }i.
  \end{equation*}
\end{proposition}
\begin{proof}    
  The finiteness assumption ensures that every quotient $E_i=E/H_i$ is a product
  \begin{equation*}
    E_i\cong \prod_{\alpha\in\Irr\bG}E_{i,\alpha}, 
  \end{equation*}
  and the conclusion follows immediately from the fact that in any category limits commute \cite[\S IX.2]{mcl_2e}:
  \begin{equation}\label{eq:isotyp-cont}
    E\cong \varprojlim_i\prod_{\alpha}E_{i,\alpha}
    \cong\prod_{\alpha}\varprojlim_i E_{i,\alpha}
    \cong \prod_{\alpha} E_{\alpha}.
  \end{equation}
\end{proof}

\begin{remarks}\label{res:notquot}
  \begin{enumerate}[(1),wide=0pt]
  \item The quotients $E_i=E/H_i$ in the statement are not necessarily equipped with the quotient topology: that is the {\it finest} possibility, but the result allows coarser options.

  \item The last isomorphism of \Cref{eq:isotyp-cont} relies implicitly on the fact that for compact $\bG$, the functor
    \begin{equation*}
      E\xmapsto{\quad}E_{\alpha},\quad \alpha\in \Irr(\bG)
    \end{equation*}

    picking out the $\alpha$-isotypic component is {\it continuous} \cite[\S V.4]{mcl_2e} (i.e. preserves arbitrary limits) on any of the various categories of TVSs equipped with $\bG$ actions. This is a composite of several simple observations:
    \begin{itemize}
    \item $E_{\alpha}$ can be recovered as the space
      \begin{equation*}
        E_\alpha = (V_{\alpha}\otimes E)^{\bG}
        ,\quad
        V_{\alpha}:=\text{carrier space of }\alpha
      \end{equation*}
      of invariants in the {\it tensor product} \cite[\S 2]{rob} of the $\bG$-representations $\alpha$ and $E$;

    \item tensoring with the finite-dimensional $V_{\alpha}$ is continuous (because \cite[Proposition 13.4]{ahs} it is easily seen to preserve products and {\it equalizers} \cite[Definition 7.51]{ahs});

    \item and taking $\bG$-invariants means forming a limit (equalizer of the maps induced by the individual $g\in \bG$) so commutes with limits \cite[\S IX.2]{mcl_2e}. 
    \end{itemize}
  \end{enumerate}  
\end{remarks}

In particular:

\begin{corollary}\label{cor:equivlim}
  Suppose $\bG\xrightarrow{\varphi}GL_s(E)$ as in \Cref{pr:equivlim} is locally $\bG$-finite and, assuming \Cref{eq:gequivelimei}, the $E_i$ are globalizing. $\varphi$ is then $\bG$-finite. 
\end{corollary}
\begin{proof}
  it is in any case the product of its isotypic components by \Cref{pr:equivlim}, so strictly larger than the {\it sum} of those isotypic components if infinitely many of them are non-zero. 
\end{proof}

\begin{remark}\label{re:strict-lim-pathological}
  A cofiltered limit \Cref{eq:elimei} is what, by analogy to the dual colimit situation \cite[\S 19.4]{k_tvs-1}, might be called {\it strict}: for indices $i\le j$ the map $E_j\to E_i$ is onto, $E_i$ is equipped with the resulting quotient topology. The terminology is also familiar in somewhat different contexts, e.g. the {\it strict pro-objects} of (or rather dual to) \cite[\S 8.12.1]{sga4_01-04}.

  Note, though, that the requirement that the structure maps
  \begin{equation*}
    E\cong \varprojlim_i E_i
    \xrightarrow{\quad}
    E_j
  \end{equation*}
  of the limit \Cref{eq:elimei} be onto is an additional requirement: this need not be so for arbitrary strict cofiltered limits, as \Cref{ex:strictlimsets} shows. 
\end{remark}

\begin{example}\label{ex:strictlimsets}
  There are trivial cofiltered limits of non-trivial (complete) LCTVSs, strict in the sense of \Cref{re:strict-lim-pathological}. One can simply adapt the analogous set-theoretic version \cite{wat_invlim} of this remark: given a cofiltered system
  \begin{equation*}
    X_j\xrightarrowdbl{\quad}X_i
    ,\quad
    i\le j
  \end{equation*}
  of set surjections with empty limit in $\cat{Set}$, consider the resulting cofiltered system 
  \begin{equation*}
    \bR^{\oplus X_j}\xrightarrowdbl{\quad}\bR^{\oplus X_i}
    ,\quad
    i\le j
  \end{equation*}
  of quotient maps, with the direct sums equipped with their locally convex direct-sum topology \cite[\S 18.5]{k_tvs-1}. The limiting LCTVS will then be trivial.
\end{example}

In light of the remarks preceding \Cref{ex:hme48-dirsum}, it seems apposite to search for a common generalization of \cite[Exercise E4.8]{hm4} and \Cref{pr:hme48-frechet} by extending those results to a natural class of LCTVSs broad enough to house both Fr\'echet and weakly complete spaces. The following result is one way to achieve that goal.  

\begin{theorem}\label{th:lim-stab-glob}
  Suppose a complete LCTVS $E$ is expressible as a cofiltered limit
  \begin{equation}\label{eq:elimei}
    E\cong \varprojlim_{i}E_i
    ,\quad
    E_i:=E/H_i\text{ complete}
  \end{equation}
  for a family of closed subspaces $H_i\le E$, filtered when ordered by reverse inclusion.

  If for each $i$ the countable power $E_i^{\aleph_0}$ is globalizing, so is $E$.
\end{theorem}

Before moving on to the proof, we record one consequence that motivated the investigation to begin with: \Cref{cor:lim-stab-glob} generalizes both \Cref{pr:hme48-frechet} (obviously) and \cite[Exercise E4.8]{hm4}, by specializing to products of copies of $\bR$ or $\bC$, which are of course Fr\'echet. 

\begin{corollary}\label{cor:lim-stab-glob}
  Arbitrary products of Fr\'echet spaces are globalizing. 
\end{corollary}
\begin{proof}
  Indeed, countable products of Fr\'echet spaces are again Fr\'echet \cite[\S 18.3(6)]{k_tvs-1} and hence globalizing by \Cref{pr:hme48-frechet}, and arbitrary products $E:=\prod_{i}E_{i}$ are expressible as cofiltered limits
  \begin{equation*}
    E\cong \varprojlim_{|\cF|<\infty} \left(E_{\cF}:=\prod_{i\in \cF}E_{i}\right)
  \end{equation*}
  meeting the requirements of \Cref{th:lim-stab-glob}. 
\end{proof}

It will be worthwhile to pick up a few more auxiliary observations along the way. Recall \cite[\S 4.9]{nb_tvs} that a(n automatically closed \cite[Theorem 4.9.3]{nb_tvs}) subspace $F\le E$ of a TVS is {\it complemented} if $E\cong F\oplus F'$  linearly and topologically. 

\begin{lemma}\label{le:compl-glob}
  The property of being ($\bG$-)globalizing passes over to complemented subspaces. 
\end{lemma}
\begin{proof}
  A $\bG$-action on a complemented subspace $F\le E$ extends over to $E$ by making $\bG$ act trivially on a complement $F'\le E$ with $E\cong F\oplus F'$, (local) $\bG$-finiteness transports over between the original action and the extension, etc.
\end{proof}

In particular, the hypothesis of \Cref{th:lim-stab-glob} now implies that {\it finite} powers of each $E_i$ are also globalizing, since they are summands in countable powers; we take this for granted. 

\pf{th:lim-stab-glob}
\begin{th:lim-stab-glob}
  One difficulty is that a jointly continuous representation $\bG\xrightarrow{\varphi}GL_s(E)$, fixed for the duration of the proof, need not be compatible with the limit realization \Cref{eq:elimei}. Part of the proof will supply that deficiency.

  \begin{enumerate}[(I),wide=0pt]

  \item {\bf Reduction to separable $\bG$.} That is, we argue here that it is enough to assume $\bG$ {\it separable} in the usual \cite[Problem 5F]{wil-top} sense that it has a dense countable subset.

    Indeed, irreducible representations being uniquely determined by their characters \cite[Corollary 4.5]{hm4}, countably infinitely many inequivalent irreducible $\bG$-representations will always return infinitely many mutually inequivalent irreducible summands if restricted to an appropriate countably-generated (hence separable) closed subgroup $\bH\le \bG$. The other relevant properties such as (local) finiteness are also preserved upon restricting from $\bG$ to $\bH$, so we may work with the latter throughout.

    We henceforth assume $\bG$ separable. 
    
  \item {\bf Reduction to $\bG$-equivariant limits \Cref{eq:elimei}.} The goal here is to show that the conclusion holds generally if it holds under the assumption that the $H_i\le E$ of \Cref{eq:elimei} are $\bG$-invariant. We do this by substituting
    \begin{equation}\label{eq:ghi}
      \tensor[_{\bG}]{H}{_i}
      :=
      \bigcap_{g\in \bG}gH_i
      \quad\text{for}\quad H_i
    \end{equation}
    respectively. Observe first that in \Cref{eq:ghi} we can always intersect over $g$ ranging over any dense subset $D\subseteq \bG$. Indeed, $H_i$ is \cite[\S 18.1(3)]{k_tvs-1} the joint kernel of those continuous {\it seminorms} \cite[\S 14.1]{k_tvs-1} $p_{s}$ on $E$ which vanish on $H_i$, so that
    \begin{equation*}
      \bigcap_{g\in \bG}{H_i}
      =
      \bigcap_{g\in \bG}\ker p_s(g-)
      =
      \bigcap_{g\in D}\ker p_s(g-)
    \end{equation*}
    by the continuity of the action: for fixed $v\in E$ and index $s$ we have
    \begin{equation*}
      p_s(gv)=0,\ \forall g\in \bG
      \quad
      \xLeftrightarrow{\quad}
      \quad
      p_s(gv)=0,\ \forall g\in \text{ dense subset }D\subseteq \bG.
    \end{equation*}
    Our separability assumption then allows us to take $D$ countable; assume also that $1\in D$ for convenience. Consider the embedding
    \begin{equation}\label{eq:eghiemb}
      E/\tensor[_{\bG}]{H}{_i}
      \lhook\joinrel\xrightarrow{\quad}
      \prod_{g\in D}E/gH_i,
    \end{equation}
    equipping its domain with the subspace topology. The image is the graph of the TVS morphism
    \begin{equation*}
      E/H_i
      \xrightarrow{\quad(1\ne g\in D)\quad}
      \prod_{1\ne g\in D}E/gH_i,
    \end{equation*}
    where
    \begin{equation*}
      E/H_i\xrightarrow[\cong]{\quad g\quad}
      E/gH_i
    \end{equation*}
    denotes the isomorphism induced by the action of $g\in G$. Said image is thus complemented (so also complete, because closed) in the product
    \begin{equation*}
      \prod_{g\in D}E/gH_i
      \cong
      E/H_i\times \prod_{1\ne g\in D}E/gH_i,
    \end{equation*}
    assumed globalizing, This makes $E/\tensor[_{\bG}]{H}{_i}$ (with the subspace topology \Cref{eq:eghiemb}) globalizing, and we are done upon observing that the limit realization \Cref{eq:elimei} also implies the analogous
    \begin{equation*}
      E\cong \varprojlim_{i}E/\tensor[_{\bG}]{H}{_i}.
    \end{equation*}
    
  \item {\bf Conclusion, assuming \Cref{eq:elimei} is $\bG$-equivariant.} Simply apply \Cref{cor:equivlim}.  \qedhere
  \end{enumerate}
\end{th:lim-stab-glob}

\begin{remark}\label{re:notquottop}
  The proof of \Cref{th:lim-stab-glob} required some care in distinguishing the quotient topology on the domain of \Cref{eq:eghiemb} from the subspace topology induced by that embedding. This is not a moot point, as the two need not coincide. The problem (of whether or not they do) is moreover linked to the issue of sums of closed subspaces failing to be closed. Consider  
  \begin{equation*}
    \text{closed }F,F'\le \text{ complete metrizable TVS }E.
  \end{equation*}
  Quotienting everything in sight by $F\cap F'$ (which will not affect complete metrizability \cite[pre \S 15.11(4)]{k_tvs-1}), we may assume that $F$ and $F'$ intersect trivially. Paraphrased, \cite[Theorem 1]{zbMATH03035405} (stated there for Banach spaces but going through in the complete metrizable case because the key ingredient is the {\it closed graph theorem} \cite[\S 12.12(3)]{k_tvs-1}) says that 
  \begin{equation*}
    F+F'
    \le 
    E
    \lhook\joinrel\xrightarrow{\quad}
    E/F\times E/F'
  \end{equation*}
  is a homeomorphism onto its image if and only if $F+F'\le E$ is closed. That it isn't always is well known (see for instance \cite[\S 15 and p.110]{halm_sm_2e_1957} for two different approaches to constructing examples of such pathologies).
\end{remark}

Consider, in closing, yet another way in which the isotypic-component machinery of \cite[Chapters 3 and 4]{hm4} can go awry, this time when local convexity is absent.

\begin{example}\label{ex:lp}
  Set $\bG:=\bS^1$ and, for fixed $0<p<1$, consider the space $E:=L^p(\bG,\mu_{Haar})$ (complex-valued functions) of \cite[\S 15.9]{k_tvs-1}. This is the preeminent example of a {\it quasi-Banach} space \cite[\S A.1]{pvl_disc} (equivalently \cite[Theorem 3.2.1]{rol_mls}, complete locally bounded TVS): complete with respect to the topology induced by
  \begin{equation*}
    \vvvert f\vvvert:= \|f\|_p^p:=\int_{\bG}|f|^p\ \mathrm{d}\mu_{Haar}
  \end{equation*}
  This is a {\it $p$-norm} in the sense of \cite[\S 15.10]{k_tvs-1}, i.e. a function satisfying the usual axioms of a norm except for homogeneity, which is instead replaced by {\it $p$-homogeneity}: $\vvvert \lambda f\vvvert =|\lambda|^p\vvvert f\vvvert$ for scalars $\lambda$. 

  $\bG$ acts by rotations $f\xmapsto{g\triangleright}f(g^{-1}\cdot)$ jointly continuously, and
  \begin{equation*}
    E_{fin}=\spn\left\{z\mapsto z^n\ |\ n\in \bZ\right\}\subset E
  \end{equation*}
  is of course dense and algebraically a direct sum of its isotypic components $E_n:=\spn\{z^n\}$, $n\in \bZ$. There are, however, no continuous spectral projections $E\xrightarrowdbl{} E_n$ as in \cite[Theorem 4.22]{hm4}, for none of the spaces $E_n$ are complemented in $E$ \cite[\S 15.9(10)]{k_tvs-1}.
\end{example}

\begin{remark}\label{re:not-diff}
  The problem in \Cref{ex:lp} stems from the fact that one would like to define said spectral projections by integration (as in \Cref{eq:averageop} for instance), and vector-valued integration is problematic in the absence of convexity. Recall first \cite[Theorem 3.5.1]{rol_mls} that the {\it Riemann integrability} \cite[\S 3.5]{rol_mls} of all $E$-valued continuous functions on a compact interval {\it implies} the local convexity of the complete metrizable TVS $E$. Sufficiently regular functions (e.g. analytic \cite[Theorem 3.5.2]{rol_mls} or more generally, sufficiently continuously differentiable \cite[Theorem 3.5.4]{rol_mls}) are integrable, but in the present case most
  \begin{equation*}
    \bS^1\ni g\xmapsto{\quad} gv\in E,\quad v\in E
  \end{equation*}
  are {\it not} continuously differentiable: simply adapt the well-known example \cite[Chapter 3, Example 8]{go_counter-analysis} of a nowhere differentiable continuous function on $[0,1]$, by extending it to a $\bZ$-periodic function on the real line and hence a piecewise-continuous function on $\bS^1\cong \bR/\bZ$.   
\end{remark}

%%%%%%%%%%%%%%%%%%%%%%%%%%%%%%%%%%%%%%%%%%%%%%%%%%%%%%%%%%%%%%%%%%%%%%%%%%%%%
%%%%%%%%%%%%%%%%%%%%%%%%%%%%%%%%%%%%%%%%%%%%%%%%%%%%%%%%%%%%%%%%%%%%%%%%%%%%%

%\newpage

\addcontentsline{toc}{section}{References}
%\bibliography{bib}{}
%\bibliographystyle{plain}

\Addresses

\end{document}